\documentclass[]{theclass}
\usepackage{graphicx, xfrac, lineno, float, subcaption, tasks, comment, xcolor, booktabs, multirow}
\usepackage[normalem]{ulem}
\usepackage{datetime}
\usepackage{colortbl}
\usepackage{enumerate}
\usepackage{hyperref}

\begin{document}
%\linenumbers

\begin{frontmatter}   %%  Title, information about author, abstract, etc.

\titledata{A note on fractional covers of a graph}           % title of the paper
{}                 % footnote on the title -- empty if not required

\authordatatwo{John Baptist Gauci}{john-baptist.gauci@um.edu.mt}{}%footnote for first author
{Jean Paul Zerafa}{zerafa.jp@gmail.com}{The research work disclosed in this publication is funded by the ENDEAVOUR Scholarship Scheme (Malta). The scholarship may be part-financed by the European Union -- European Social Fund (ESF) under Operational Programme II -- Cohesion Policy 2014--2020, ``Investing in human capital to create more opportunities and promote the well being of society".}
{Department of Mathematics, University of Malta, Malta}


\begin{abstract}
A fractional colouring of a graph $G$ is a function that assigns a non-negative real value to all possible colour-classes of $G$ containing any vertex of $G$, such that the sum of these values is at least one for each vertex. The fractional chromatic number is the minimum sum of the values assigned by a fractional colouring over all possible such colourings of $G$. Introduced by Bosica and Tardif, fractional covers are an extension of fractional colourings whereby the real-valued function acts on all possible subgraphs of $G$ belonging to a given class of graphs. The fractional chromatic number turns out to be a special instance of the fractional cover number. In this work we investigate fractional covers acting on \linebreak$(k+1)$-clique-free subgraphs of $G$ which, although sharing some similarities with fractional covers acting on $k$-colourable subgraphs of $G$, they exhibit some peculiarities. We first show that if a simple graph $G_2$ is a homomorphic image of a simple graph $G_1$, then the fractional cover number defined on the $(k+1)$-clique-free subgraphs of $G_1$ is bounded above by the corresponding number of $G_2$. We make use of this result to obtain bounds for the associated fractional cover number of graphs that are either $n$-colourable or \linebreak$a\!\!:\!\!b$-colourable.
\end{abstract}

\end{frontmatter}   %% End of the front matter

%% Your article

\section{Introduction and terminology}

Let $G$ be a simple graph (with no loops or multiple edges) having vertex set $V(G)$ and edge set $E(G)$. The \emph{clique number} $\omega(G)$ of $G$ is the cardinality of a largest set of vertices of $G$ which induces a complete subgraph. Recall that if $H$ is an induced subgraph of $G$, then $V(H)\subseteq V(G)$ and $uv \in E(H)$ for any $u,v \in V(H)$ whenever $uv \in E(G)$. On the other extreme, a set of vertices of $V(G)$ no two of which are adjacent is an \emph{independent set}. We let $\mathbb{F}$ denote a class of graphs, closed under isomorphism. The set of induced subgraphs of $G$ belonging to the class of graphs $\mathbb{F}$ is denoted by $\mathbb{F}(G)$. For instance, if $\mathbb{F}$ is the class of all bipartite graphs, then, given a graph $G$, $\mathbb{F}(G)$ is the set containing all possible induced subgraphs of $G$ which are bipartite, possibly even $G$ itself (if $G$ is bipartite). The elements of $\mathbb{F}(G)$ which contain $v$, for some $v \in V(G)$, are denoted by $\mathbb{F}(G,v)$. In particular, we use $\mathbb{I}(G)$ to denote the set of all independent sets of $G$, and $\mathbb{I}(G,v)$ to denote all the independent sets of $G$ that contain $v$, for some $v\in V(G)$.

A (proper) \emph{vertex-colouring} of $G$ is an assignment of colours to the vertices of $G$ such that adjacent vertices receive different colours. If $G$ can be properly coloured using $n$ colours, we say that $G$ is \emph{$n$-colourable}. The \emph{chromatic number} $\chi(G)$ of $G$ is the least number of colours required such that $G$ has a proper colouring. A variant of the chromatic number is the fractional chromatic number, where in fractional colouring the vertices of a graph $G$ are assigned a set of colours instead of a single colour. More formally, the definitions of fractional colouring and fractional chromatic number that we adopt are those used in \cite[pp.135--136]{GodsilRoyle} and presented here in Definition \ref{Definition:FractionalColouring}.

\begin{definition}\label{Definition:FractionalColouring}
A \emph{fractional colouring} of $G$ is a non-negative real-valued function $f$ on $\mathbb{I}(G)$ such that for any vertex $v \in V(G)$
\begin{linenomath}
$$\sum_{I\in \mathbb{I}(G,v)}f(I)\geq 1.$$
\end{linenomath}
The \emph{weight} of a fractional colouring $f$ is equal to $\sum_{I \in \mathbb{I}(G)}f(I)$. The \emph{fractional chromatic number} is given by
\begin{linenomath}
$$\chi_{\textrm{f}}(G)=\min_{f} \sum_{I \in \mathbb{I}(G)}f(I),$$
\end{linenomath}
where the minimum is taken over all possible fractional colourings $f$ of $G$.
\end{definition}

We remark that although the empty set $\emptyset$ is an independent set, $\sum_{I \in \mathbb{I}(G,v)}f(I)$ is not affected by $f(\emptyset)$, as $\emptyset \not\in \mathbb{I}(G,v)$ for all $v \in V(G)$. Thus, $f(\emptyset)$ only affects the weight of $f$ by making it greater, and since the aim is to minimise the weight of $f$, we shall define $f(\emptyset)=0$. Hence, in the sequel, we can either assume that $f(\emptyset)=0$, or, alternatively, consider only non-empty independent sets. Similarly, assigning values to $f(I)$ which are greater than one has only an adverse effect on the fractional chromatic number and we thus limit the range of $f$ to the interval $\left[0,1\right]$.

One interesting problem in graph colouring is the Erd\H{o}s--Faber--Lov\'{a}sz conjecture, first formulated in 1972 \cite{Erdos1981}. This conjecture states that if a graph $G$ is the union of $n$ cliques each of order $n$ with the condition that no two of them share more than one vertex, then $\chi(G)=n$. In January 2021, Kang \emph{et al.} \cite{kuhn} announced that the conjecture is true for sufficiently large values of $n$. The fractional analogue was completely proved by Kahn and Seymour in 1992 \cite{KahnSeymour1992}. We refer the reader to \cite{GodsilRoyle, FractionalBook} for more results on fractional colourings and on the fractional chromatic number.

Motivated by the above, Bosica and Tardif \cite{BosicaThesis, BosicaTardif} defined the fractional $\mathbb{F}\text{-cover}$ of a graph $G$, given here in Definition \ref{definition:fractionalFcover}.

\begin{definition}\cite{BosicaTardif}\label{definition:fractionalFcover}
A \emph{fractional $\mathbb{F}$}-cover of $G$ is a function $f:\mathbb{F}(G)\rightarrow \left[0,1\right]$ such that for any vertex $v \in V(G)$
\begin{linenomath}
$$\sum_{H \in \mathbb{F}(G,v)}f(H) \geq 1.$$
\end{linenomath}
The \emph{weight} of a fractional $\mathbb{F}$-cover $f$ is equal to $\sum_{H \in \mathbb{F}(G)}f(H)$. The \emph{fractional $\mathbb{F}$}-cover \emph{number} of $G$, denoted by $\mathbb{F}$-cover$_{\textrm{f}}(G)$, is the minimum possible weight over all fractional $\mathbb{F}$-covers of $G$, that is,
\begin{linenomath}
$$\mathbb{F}\textrm{-cover}_{\textrm{f}}(G)=\min_{f}\sum_{H \in \mathbb{F}(G)}f(H).$$
\end{linenomath}
\end{definition}

In particular, when $\mathbb{F}(G)$ is the set of all subgraphs of $G$ induced by non-empty independent sets, then $\mathbb{I}$-cover$_{\text{f}}(G)=\chi_{\text{f}}(G)$. We shall have occasion to consider the class $\mathbb{K}_{k}$ composed of graphs not containing a $(k+1)$-clique (referred to as $(k+1)$-clique-free graphs), and the class $\mathbb{C}_{k}$ composed of graphs which are $k$-colourable. As noted in \cite{BosicaTardif}, one can easily see that $\mathbb{I}(G)= \mathbb{K}_{1}(G)=\mathbb{C}_{1}(G)$, implying that $\chi_{\text{f}}(G) = \mathbb{C}_{1}\text{-cover}_{\text{f}}(G) = \mathbb{K}_{1}\text{-cover}_{\text{f}}(G)$. For more general results about the fractional $\mathbb{F}$-cover of a graph $G$, its dual (that is, the fractional $\mathbb{F}$-clique), and how these can be used when dealing with the Erd\H{o}s--Faber--Lov\'{a}sz Conjecture, the reader is referred to \cite{BosicaTardif}.

In the next section we start by presenting some bounds obtained by Bosica and Tardif \cite{BosicaTardif} for fractional $\mathbb{C}_{k}$- and $\mathbb{K}_{k}$-cover numbers. Our main result is given in Lemma \ref{Lemma:Homomorphism}. We show that, given a class of graphs $\mathbb{F}$, if there exists a homomorphism $\theta$ from $G_{1}$ to $G_{2}$ such that the preimage under $\theta$ of each subgraph in $\mathbb{F}(G_{2})$ is also a subgraph of $\mathbb{F}(G_{1})$, then $\mathbb{F}$-cover$_{\rm{f}}(G_{1})\leq \mathbb{F}$-cover$_{\rm{f}}(G_{2})$. As a consequence of this, in  Corollary \ref{Corollary:Homomorphism} we show that if there exists a homomorphism $\theta$ from $G_{1}$ to $G_{2}$, then, $\mathbb{K}_{k}$-cover$_{\text{f}}(G_{1}) \leq \mathbb{K}_{k}$-cover$_{\text{f}}(G_{2})$, for any $k\geq 1$. In Section \ref{SectionColourings}, we find some upper bounds for the fractional $\mathbb{K}_{k}$-cover number of graphs which are $n\text{-colourable}$ and of graphs which are $a\!\!:\!\!b$-colourable. These bounds are discussed in Theorems \ref{Theorem1:n-colourable} and \ref{Theorem2:ab-colourable}, respectively.

\section{Bounds for fractional covers and main result}\label{section example}

Given that the vertex set of a graph $G$ is sufficiently large and since  $\mathbb{C}_{k} \subseteq \mathbb{C}_{k+1}$, from \cite{BosicaTardif} we know that the sequence $\big(\mathbb{C}_{k}\text{-cover}_{\text{f}}(G)\big)_{1 \leq k \leq \chi(G)}$ is non-increasing with \linebreak$\mathbb{C}_{\chi(G)}\text{-cover}_{\text{f}}(G)=1$. Thus
\begin{linenomath}
$$1=\mathbb{C}_{\chi(G)}\text{-cover}_{\text{f}}(G)\leq \ldots \leq \mathbb{C}_{1}\text{-cover}_{\text{f}}(G)=\chi_{\text{f}}(G).$$
\end{linenomath}
For $k> \chi(G)$, $\mathbb{C}_{k}\text{-cover}_{\text{f}}(G)$ is defined to be equal to one. Bosica and Tardif also show that $\big(k\cdot\mathbb{C}_{k}\text{-cover}_{\text{f}}(G)\big)_{1 \leq k \leq \chi(G)}$ is a non-decreasing sequence. Thus
\begin{linenomath}
\begin{equation*}
\begin{split}
\chi_{\text{f}}(G)=1\cdot \mathbb{C}_{1}\text{-cover}_{\text{f}}(G) & \leq 2\cdot \mathbb{C}_{2}\text{-cover}_{\text{f}}(G)\leq \ldots \\
& \leq \chi(G)\cdot \mathbb{C}_{\chi(G)}\text{-cover}_{\text{f}}(G)= \chi(G).
\end{split}
\end{equation*}
\end{linenomath}

For any integer $k$, $\mathbb{C}_k \subseteq \mathbb{K}_k$ (because any $k$-colourable graph does not contain
a clique of order $k+1$), but $\mathbb{K}_{k}\not \subseteq \mathbb{C}_{k}$. For example, the cycle  $C_5$ on five vertices is
triangle free, that is, $C_5$ belongs to $\mathbb{K}_2$ but not to $\mathbb{C}_{2}$. This implies that $\mathbb{K}_k\text{-cover}_{\text{f}}(G) \leq
\mathbb{C}_k\text{-cover}_{\text{f}}(G)$, and for $1 \leq i \leq \omega(G)$ we have \begin{linenomath}
$$\mathbb{K}_{i}\text{-cover}_{\text{f}}(G)\leq \mathbb{C}_{i}\text{-cover}_{\text{f}}(G)\leq \mathbb{C}_{i-1}\text{-cover}_{\text{f}}(G) \leq \ldots \leq \mathbb{C}_{1}\text{-cover}_{\text{f}}(G).$$
\end{linenomath}

Since $\mathbb{K}_{k}\subseteq \mathbb{K}_{k+1}$, $\big(\mathbb{K}_{k}\text{-cover}_{\text{f}}(G)\big)_{1\leq k \leq \omega(G)}$ is also a non-increasing sequence with $\mathbb{K}_{\omega(G)}\text{-cover}_{\text{f}}(G)=1$. For $k> \omega(G)$, $\mathbb{K}_{k}\text{-cover}_{\text{f}}(G)$ is defined to be equal to one. Despite this similarity with $\mathbb{C}_{k}$-covers, the sequence  $\big(k \cdot \mathbb{K}_{k}\text{-cover}_{\text{f}}(G)\big)_{1\leq k \leq\omega(G)}$ is not necessarily non-decreasing. For example, in \cite{BosicaThesis}, it is stated that when $G=C_{5}* C_{5}$, where $*$ denotes the co-normal product between graphs, $\big(k \cdot \mathbb{K}_{k}\text{-cover}_{\text{f}}(G)\big)_{1\leq k \leq 5}=(\frac{25}{4},5,\frac{75}{14},5,5)$.
It is due to this peculiarity that we investigate further the fractional $\mathbb{K}_k$-cover numbers. We are now in a position to prove our main result.

\begin{lemma}\label{Lemma:Homomorphism}
Let $G_1$ and $G_2$ be two simple graphs and let $\mathbb{F}$ be a class of graphs. If $\theta$ is a homomorphism from $G_{1}$ to $G_{2}$ such that the preimage (under $\theta$) of each subgraph in $\mathbb{F}(G_{2})$ is also a subgraph of $\mathbb{F}(G_{1})$, then
\begin{linenomath}
$$\mathbb{F}\emph{-cover}_{\rm{f}}(G_{1})\leq \mathbb{F}\emph{-cover}_{\rm{f}}(G_{2}).$$
\end{linenomath}
\end{lemma}

\begin{proof} Let $f_{2}$ be a minimum fractional $\mathbb{F}$-cover of $G_{2}$, that is,
\begin{linenomath}
$$\mathbb{F}\text{-cover} _{\text{f}}(G_{2})=\sum_{H \in \mathbb{F}(G_{2})}f_{2}(H).$$
\end{linenomath}
We define $f_{1}: \mathbb{F}(G_{1}) \rightarrow \left[0,1\right]$ such that, for $J \in \mathbb{F}(G_1)$,

\begin{linenomath}
$$f_{1}(J)= \left\{
\begin{array}{cl}
\min\left\{1,\sum'f_{2}(H)\right\} & $if $ J=H',\\
0 & $otherwise,$
\end{array}\right.$$
\end{linenomath}
where $\sum'$ is the sum taken over all those $H\in\mathbb{F}(G_{2})$ which have $H'$ as a preimage under $\theta$. The well-definition of $f_{1}$ follows immediately. Let $u \in V(G_{1})$. If $u$ is in the vertex set of some $J^* \in \mathbb{F}(G_{1})$ such that $f_{1}(J^*)=1$, then
\begin{linenomath}
$$\sum_{J\in \mathbb{F}(G_{1},u)}f_{1}(J)\geq f_{1}(J^*)=1.$$
\end{linenomath}
Otherwise,
\begin{linenomath}
$$\sum_{J\in \mathbb{F}(G_{1},u)}f_{1}(J)= \sum_{H'\in \mathbb{F}(G_{1},u)}f_{1}(H')=\sum_{H\in \mathbb{F}(G_{2},\theta(u))}f_{2}(H) \geq 1,$$
\end{linenomath}
where the subgraphs $H'\in \mathbb{F}(G_{1},u)$ are those mapped to a non-zero value under $f_{1}$. Such subgraphs $H'$ exist since $f_{2}$ is a fractional $\mathbb{F}$-cover of $G_{2}$ and the preimage (under $\theta$) of each subgraph in $\mathbb{F}(G_{2})$ is also a subgraph of $\mathbb{F}(G_{1})$. Therefore, $f_{1}$ is a fractional $\mathbb{F}$-cover of $G_{1}$ with weight at most that of $f_{2}$, implying
\begin{linenomath}
$$\mathbb{F}\text{-cover}_{\text{f}}(G_{1})\leq \sum_{J \in \mathbb{F}(G_{1})} f_{1}(J)\leq\sum_{H \in \mathbb{F}(G_{2})} f_{2}(H)=\mathbb{F}\text{-cover}_{\text{f}}(G_{2}),$$
\end{linenomath}
proving our lemma.
\end{proof}

\begin{corollary}\label{Corollary:Homomorphism}
Let $G_1$ and $G_2$ be two simple graphs. If $\theta$ is a homomorphism from $G_{1}$ to $G_{2}$, then
\begin{linenomath}
$$\mathbb{K}_{k}\emph{-cover}_{\rm{f}}(G_{1})\leq \mathbb{K}_{k}\emph{-cover}_{\rm{f}}(G_{2}),$$
\end{linenomath}
for any $k\geq 1$.
\end{corollary}

\begin{proof} Let $\theta$ be a homomorphism from $G_{1}$ to $G_{2}$. By Lemma \ref{Lemma:Homomorphism}, it suffices to show that the preimage (under $\theta$) of each subgraph in $\mathbb{K}_{k}(G_{2})$ is also a subgraph of $\mathbb{K}_k(G_{1})$. Let $V(G_{2})=\{v_{1},\ldots, v_{n}\}$, and if $v_{i}$ is in the range of $\theta$, we define $U_{i}$ such that $U_{i}=\{u \in V(G_{1}): \theta(u)=v_{i}\}$. If $v_{i}$ is not in the range of $\theta$ (that is, $\theta$ is not onto), then we let $U_{i}=\emptyset$. The set $U_{i}$ of vertices of $G_1$ is an independent set for each $i\in \{1,\ldots, n\}$.
Let $H \in \mathbb{K}_{k}(G_{2})$ and suppose $V(H)=\{v_{1},\ldots, v_{m}\}$. Define $H'$ to be the subgraph of $G_{1}$ induced by the vertices in $U_{1}\cup \ldots \cup U_{m}$. We claim that $H' \in \mathbb{K}_{k}(G_{1})$. The result follows immediately when $m\leq k$. For $m>k$, suppose that $H'$ contains a $(k+1)$-clique induced by the vertices $u_{1}, \ldots, u_{k+1}$. Note that each vertex $u_{i}$ is an element of a different $U_{i}$, where $i \in \{1,\ldots, m\}$, and by the definition of $U_{i}$, the vertices $\theta(u_{1}) \ldots, \theta(u_{k+1})$ are all distinct and induce a $(k+1)$-clique in $H$. This contradicts the fact that $H \in \mathbb{K}_{k}(G_{2})$, and so does not contain $(k+1)$-cliques. Thus, $H' \in \mathbb{K}_k(G_1)$.
\end{proof}

\section{Graph colourings and fractional $\mathbb{K}_{k}$-covers}\label{SectionColourings}

In this section we obtain an upper bound for the fractional $\mathbb{K}_{k}$-cover number of graphs either having an $n$-colouring or an $a\!\!:\!\!b$-colouring. We start with $n\text{-colourable}$ graphs.
\begin{theorem}\label{Theorem1:n-colourable}
If $G$ is $n$-colourable, then $\mathbb{K}_{k}$\emph{-cover}$_{\rm{f}}(G)\leq \frac{n}{k}$, for $1 \leq k \leq \omega(G)$.
\end{theorem}

\begin{proof}[Proof 1]
Since $\mathbb{C}_k \subseteq \mathbb{K}_k$, we know that $\mathbb{K}_k\text{-cover}_{\text{f}}(G)\leq \mathbb{C}_k\text{-cover}_{\text{f}}(G)$, implying that $k\cdot\mathbb{K}_k\text{-cover}_{\text{f}}(G)\leq k\cdot\mathbb{C}_k\text{-cover}_{\text{f}}(G)$. Also, since $\big(k\cdot\mathbb{C}_{k}\text{-cover}_{\text{f}}(G)\big)_{1 \leq k \leq \chi(G)}$ is a non-decreasing sequence reaching $\chi(G)$, then $k\cdot\mathbb{K}_k\text{-cover}_{\text{f}}(G)\leq\chi(G)$. Therefore, if $G$ is $n$-colourable, then $\chi(G)\leq n$, and for $1\leq k\leq \omega(G)$, we obtain $\mathbb{K}_{k}\text{-cover}_{\text{f}}(G)\leq \frac{n}{k}.$
\end{proof}

We state a result which was proved in \cite{BosicaThesis} about vertex-transitive graphs.

\begin{theorem}\cite{BosicaThesis}\label{Theorem:VertexTransitive}
If $G$ is a vertex-transitive graph then
\begin{linenomath}
$$\mathbb{K}_{k}\emph{-cover}_{\rm{f}}(G)=\frac{\vert V(G)\vert}{\beta_{k}(G)},$$
\end{linenomath}
where $\beta_{k}(G)$ is the maximum number of vertices in a $(k+1)$-clique-free subgraph of $G$.
\end{theorem}

\begin{remark}\label{Remark:beta1&omega}
Note that $\beta_{1}(G)$ is the cardinality of a largest independent set of $G$, and $\beta_{\omega(G)}=\vert V(G) \vert$.
\end{remark}

Theorem \ref{Theorem:VertexTransitive} can be used together with Corollary \ref{Corollary:Homomorphism} to give an alternative proof to Theorem \ref{Theorem1:n-colourable}.

\begin{proof}[Proof 2]
If $G$ is $n$-colourable, then it is well-known that there exists a homomorphism from $G$ to $K_{n}$, the complete graph on $n$ vertices.
By Theorem \ref{Theorem:VertexTransitive}, $\mathbb{K}_{k}\text{-cover}_{\text{f}}(K_{n})=\frac{n}{k}$, for $1\leq k \leq n$. Therefore, by Corollary \ref{Corollary:Homomorphism}, \begin{linenomath}$$\mathbb{K}_{k}\text{-cover}_{\text{f}}(G)\leq \mathbb{K}_{k}\text{-cover}_{\text{f}}(K_{n})=\frac{n}{k},$$\end{linenomath} for $1\leq k\leq\omega(G)\leq n$.
\end{proof}

We note that although the first proof is elegant, the significance of the second proof is in that it shows that the upper bound is actually sharp. Next we consider graphs having an $a\!\!:\!\!b$-colouring, defined as follows.

\begin{definition}
A graph $G$ has an \emph{$a\!\!:\!\!b$-colouring} if its vertices are assigned a set of $b$ colours taken from a palette of $a$ colours, and adjacent vertices are given disjoint sets of colours.
\end{definition}

As already mentioned, a graph $G$ is $n$-colourable if and only if there is a graph homomorphism from $G$ to $K_n$. Similarly, $G$ has an $a\!\!:\!\!b$-colouring if and only if there is a graph homomorphism from $G$ to the Kneser graph $KG(a,b)$ \cite[p.32]{FractionalBook}, where $KG(a,b)$ is defined as follows.

\begin{definition}\label{definition:kneser}
Let $a,b \in \mathbb{Z}^{+}$ and let $\left[a\right]$ denote the set $\{1,2,\ldots,a\}$. The \emph{Kneser graph} $KG(a,b)$ is the graph on the set of vertices $V\big(KG(a,b)\big)=\binom{[a]}{b}$, with two vertices being adjacent if they are disjoint $b$-subsets.
\end{definition}

In the sequel, we only consider the case when $a\geq2b$ because otherwise $KG(a,b)$ has no edges. In 1955, Kneser \cite{Kneser} conjectured that $\chi\big(KG(a,b)\big)=a-2b+2$, and this was proved in 1978 by Lov\'{a}sz \cite{Lovasz}. Therefore, by Corollary \ref{Corollary:Homomorphism}, $\mathbb{K}_k\text{-cover}_{\text{f}}\big(KG(a,b)\big)\leq \mathbb{K}_k\text{-cover}_{\text{f}}(K_{a-2b+2})=\frac{a-2b+2}{k}$, for ${1\leq k\leq\omega\big(KG(a,b)\big)}$. We can also say that if $G$ is $a\!\!:\!\!b$-colourable, then $\mathbb{K}_k\text{-cover}_{\text{f}}(G)\leq \frac{a-2b+2}{k}$, for $1\leq k \leq \omega(G)$. However, since Kneser graphs are vertex-transitive, we could use Theorem \ref{Theorem:VertexTransitive} to find the value of $\mathbb{K}_k\text{-cover}_{\text{f}}\big(KG(a,b)\big)$. The order of $KG(a,b)$ is $\binom{a}{b}$, so we only need to find the values of $\beta_{k}\big(KG(a,b)\big)$ for $1<k<\omega\big(KG(a,b)\big)$ since, by Remark \ref{Remark:beta1&omega}, $\beta_{1}\big(KG(a,b)\big)=\binom{a-1}{b-1}$ from the Erd\H{o}s--Ko--Rado Theorem and $\beta_{\omega(KG(a,b))}\big(KG(a,b)\big)=\binom{a}{b}$. This also implies that $\mathbb{K}_1\text{-cover}_{\text{f}}\big(KG(a,b)\big)=\frac{a}{b}$ and $\mathbb{K}_{\omega(KG(a,b))}\text{-cover}_{\text{f}}\big(KG(a,b)\big)=1$.

\begin{remark}\label{Remark:Baranyai}
The non-trivial values of $\beta_k\big(KG(a,b)\big)$ are an open problem in extremal combinatorics. First note that if a set of vertices in $KG(a,b)$ induces a clique, then the $b$-subsets of $[a]$ corresponding to these vertices must be pairwise disjoint. The
largest collection of pairwise disjoint $b$-subsets of $[a]$ has size $\lfloor\frac{a}{b}\rfloor$, and so $\omega\big(KG(a,b)\big)=\left\lfloor \frac{a}{b}\right\rfloor$. Moreover, the Kneser graph $KG(a,b)$ contains all Kneser graphs $KG(a',b)$, for $a\geq a'\geq 2b$. In particular, for ${1<k<\omega\big(KG(a,b)\big)}$, the largest $a'$ such that $KG(a',b)$ is a $K_{k+1}$-free subgraph of $KG(a,b)$ is equal to the largest $a'$ such that $\big\lfloor\frac{a'}{b}\big\rfloor=k$, namely $a'=(k+1)b-1$. Therefore, $\beta_k\big(KG(a,b)\big)\geq\binom{(k+1)b-1}{b}$.
\end{remark}

\begin{example}\label{Example:beta}
Consider $KG(6,2)$, with clique number 3. The ten vertices that induce $KG(5,2)$ (that is, the Petersen graph) do not induce a $K_3$. Therefore, $\beta_2\big(KG(6,2)\big)\geq 10$. However, if we consider $KG(10,2)$, despite $KG(5,2)$ being the largest $K_3$-free Kneser graph (with the parameter $b$ equal to 2) which is a subgraph of $KG(10,2)$, \linebreak$\beta_2\big(KG(10,2)\big)\geq 17$. In fact, the set of vertices $\{v\in V\big(KG(10,2)\big):v \cap \left[2\right] \neq \emptyset\}$ of cardinality $17$ does not induce a $K_3$. In general, for $1\leq k\leq \omega\big(KG(a,b)\big)$, the set $\{v\in V\big(KG(a,b)\big): v \cap \left[k\right] \neq \emptyset\}$ of cardinality $\binom{a}{b}-\binom{a-k}{b}$ does not induce a $K_{k+1}$ subgraph of $KG(a,b)$. Later on we shall see that $\beta_2\big(KG(6,2)\big)$ and $\beta_2\big(KG(10,2)\big)$ are actually equal to 10 and 17, respectively.
\end{example}

From Remark \ref{Remark:Baranyai} and an argument similar to that used in Example \ref{Example:beta} the following lemma can be deduced.

\begin{lemma}\label{Lemma:betak}
Let $1<k<\left\lfloor \frac{a}{b}\right\rfloor=\omega\big(KG(a,b)\big)$. Then
\begin{linenomath}
$$\beta_{k}\big(KG(a,b)\big)\geq\max\left\{\binom{(k+1)b-1}{b}, \binom{a}{b}-\binom{a-k}{b}\right\}.$$
\end{linenomath}
\end{lemma}

Equality does not necessarily follow as this is still an unsolved problem, known as the Erd\H{o}s Matching Conjecture, which was made in 1965 by Erd\H{o}s \cite{Erdos1965} and is stated hereunder.

\begin{conjecture}\cite{Erdos1965}\label{Conjecture:MatchingErdos}
Let $\mathcal{F}$ be a family of $b$-subsets of $\left[a\right]$, containing no $k+1$ pairwise disjoint members. Then
\begin{linenomath}
$$\vert \mathcal{F}\vert \leq \max\left\{\binom{(k+1)b-1}{b}, \binom{a}{b}-\binom{a-k}{b}\right\},$$
\end{linenomath}
for $a\geq(k+1)b-1$.
\end{conjecture}

Since $KG(a,1)$ is the complete graph on $a$ vertices and $\beta_k(K_a)=k$ for $1\leq k\leq a$, all the possible values of $\beta_k\big(KG(a,1)\big)$ are known. In 1959, Erd\H{o}s and Gallai \cite{ErdosGallai} proved that Conjecture \ref{Conjecture:MatchingErdos} is true when $b = 2$, which justifies equality in the cases considered in Example \ref{Example:beta}. In 2016, Frankl \cite{Frankl2017} proved that the conjecture is true when $b=3$. For more recent results and bounds regarding this conjecture we suggest \cite{Frankl2013} and \cite{franklarxiv}. We can now present the upper bound for the fractional $\mathbb{K}_k$-cover number of $a\!\!:\!\!b$-colourable graphs.

\begin{theorem}\label{Theorem2:ab-colourable}
If $G$ is $a\!\!:\!\!b$-colourable, then,

\begin{enumerate}[(1)]
\item $\mathbb{K}_{1}$\emph{-cover}$_{\rm{f}}(G)\leq \frac{a}{b}$,
\item $\mathbb{K}_{k}$\emph{-cover}$_{\rm{f}}(G)\leq \frac{\binom{a}{b}}{\max\left\{\binom{(k+1)b-1}{b}, \binom{a}{b}-\binom{a-k}{b}\right\}}$, for $1 < k < \omega(G)$, and
\item $\mathbb{K}_{\omega(G)}$\emph{-cover}$_{\rm{f}}(G)=1$.
\end{enumerate}
\end{theorem}

\begin{proof}
Since $G$ is $a\!\!:\!\!b$-colourable, there exists a homomorphism from $G$ to $KG(a,b)$. Therefore, by Corollary \ref{Corollary:Homomorphism}, $\mathbb{K}_{1}\text{-cover}_{\text{f}}(G)\leq \mathbb{K}_{1}\text{-cover}_{\text{f}}\left(KG(a,b)\right)=\frac{a}{b}$, proving $(1)$.
Before proving $(2)$ we note that for any two graphs $G_{1}$ and $G_{2}$, if there exists a homomorphism from $G_{1}$ to $G_{2}$, then $\omega(G_{1})\leq \omega(G_{2})$. Hence,
\begin{linenomath}
$$\omega(G)\leq \omega\big(KG(a,b)\big)=\left\lfloor \frac{a}{b}\right\rfloor.$$
\end{linenomath}
By Corollary \ref{Corollary:Homomorphism}, Theorem \ref{Theorem:VertexTransitive} and Lemma \ref{Lemma:betak}
\begin{linenomath}
$$\mathbb{K}_{k}\text{-cover}_{\text{f}}(G)\leq \mathbb{K}_{k}\text{-cover}_{\text{f}}\big(KG(a,b)\big)\leq\frac{\binom{a}{b}}{\max\left\{\binom{(k+1)b-1}{b}, \binom{a}{b}-\binom{a-k}{b}\right\}},$$
\end{linenomath}
for $1<k<\omega(G)$, proving $(2)$.
Statement $(3)$ follows immediately.
\end{proof}

\begin{corollary}
For any simple and $a\!\!:\!\!b$-colourable graph $G$, $\omega(G)\leq \left\lfloor \frac{a}{b} \right\rfloor$.
\end{corollary}

Note that Theorem \ref{Theorem1:n-colourable} is implied by Theorem \ref{Theorem2:ab-colourable} by taking $a=n$ and $b=1$.

\section{Problems}
The problems we propose relate to the sequence $\big(k\cdot\mathbb{K}_{k}$-cover$_{\text{f}}(G)\big)_{1 \leq k \leq \omega(G)}$. As mentioned in Section \ref{section example}, there are instances when this sequence is not monotonic and so it would be interesting to see which graphs yield a monotonic sequence.

\begin{problem}\label{ProblemMonotonic}
Which are the classes of graphs $\mathbb{F}$ such that for every $G\in \mathbb{F}$, the sequence  $\left(k\cdot\mathbb{K}_{k}\textrm{-cover}_\textrm{f}(G)\right)_{1 \leq k \leq \omega(G)}$ is monotonic?
\end{problem}

Empirical evidence seems to suggest that a complete answer to Problem \ref{ProblemMonotonic} is still elusive. A more manageable problem might be the following.

\begin{problem}\label{ProblemConstant}
Which are the classes of graphs $\mathbb{F}$ such that for every $G\in \mathbb{F}$, the sequence
$\big(k\cdot\mathbb{K}_{k}\textrm{-cover}_\textrm{f}(G)\big)_{1 \leq k \leq \omega(G)}$ is constant?
\end{problem}

One trivial class of graphs that answers Problem \ref{ProblemConstant} is the class of complete graphs $K_{n}$. For these graphs, $\big(k\cdot\mathbb{K}_{k}$-cover$_{\text{f}}(K_{n})\big)_{1 \leq k \leq \omega(K_{n})}$ is always equal to $n$. Bipartite graphs form another class that answers Problem \ref{ProblemConstant}. In general, if $\big(k\cdot\mathbb{K}_{k}$-cover$_{\text{f}}(G)\big)_{1 \leq k \leq \omega(G)}$ is constant, we must surely have that
\begin{linenomath}$$\chi_{\text{f}}(G)=1\cdot\mathbb{K}_{1}\text{-cover}_{\text{f}}(G)=\omega(G)\cdot\mathbb{K}_{\omega(G)}\text{-cover}_{\text{f}}(G)=\omega(G).$$
\end{linenomath}
Thus, perfect graphs could be a plausible class of graphs which answers Problem \ref{ProblemConstant}, since a graph $G$ is called \emph{perfect} if for any induced subgraph $H$ of $G$, $\omega(H)=\chi(H)$. For any graph $G$, $\omega(G)\leq \chi_{\text{f}}(G)\leq \chi(G)$ \cite[p.145]{GodsilRoyle}, and so if $G$ is perfect, then $\chi_{\text{f}}(G)=\omega(G)$. Complete graphs and bipartite graphs are perfect graphs, but are all perfect graphs a part of the solution of Problem \ref{ProblemConstant}?

\end{document}